\documentclass{amsart}

\usepackage{hyperref, amssymb}
\usepackage{color}
\usepackage[all]{xy}

\theoremstyle{plain}
\newtheorem{theorem}{Theorem}[section]
\newtheorem{proposition}[theorem]{Proposition}
\newtheorem{lemma}[theorem]{Lemma}
\newtheorem{corollary}[theorem]{Corollary}

\theoremstyle{definition}

\newtheorem{remark}[theorem]{Remark}

\newcommand{\f}{\varphi}

\newcommand{\CC}{\mathbb C}
\newcommand{\PP}{\mathbb P}
\newcommand{\ZZ}{\mathbb Z}
\newcommand{\MM}{\mathbf M}

\newcommand{\A}{{\mathcal A}}
\newcommand{\E}{{\mathcal E}}
\newcommand{\F}{{\mathcal F}}
\newcommand{\G}{{\mathcal G}}
\newcommand{\I}{{\mathcal I}}
\newcommand{\J}{{\mathcal J}}
\def\L{\mathcal L}
\def\O{\mathcal O}

\newcommand{\Ker}{{\mathcal Ker}}
\newcommand{\Coker}{{\mathcal Coker}}
\newcommand{\Image}{{\mathcal Im}}

\newcommand{\Aut}{\operatorname{Aut}}
\newcommand{\Ext}{\operatorname{Ext}}
\newcommand{\Hilb}{\operatorname{Hilb}}
\newcommand{\Hom}{\operatorname{Hom}}
\def\H{\operatorname{H}}
\newcommand{\M}{\operatorname{M}}

\newcommand{\Poly}{\operatorname{P}}
\newcommand{\p}{\operatorname{p}}
\newcommand{\rank}{\operatorname{rank}}
\newcommand{\Sym}{\operatorname{S}}

\newcommand{\dual}{{\scriptscriptstyle \operatorname{D}}}
\newcommand{\trans}{{\scriptscriptstyle \operatorname{T}}}

\newcommand{\tensor}{\otimes}

\newcommand{\lra}{\longrightarrow}

\begin{document}

\title[Moduli of sheaves supported on curves of genus $2$ in $\PP^1 \times \PP^1$]
{Moduli of sheaves supported on curves of genus two in a quadric surface}

\author{Mario Maican}
\address{Institute of Mathematics of the Romanian Academy, Calea Grivitei 21, Bucharest 010702, Romania}

\email{maican@imar.ro}

\begin{abstract}
We study the moduli space of stable sheaves of Euler characteristic $1$,
supported on curves of arithmetic genus $2$ contained in a smooth quadric surface.
We show that this moduli space is rational.
We give a classification of the stable sheaves involving locally free resolutions or extensions.
We compute the Betti numbers by studying the variation of the moduli spaces of $\alpha$-semi-stable pairs.
\end{abstract}


\maketitle

\section{Introduction}
\label{introduction}

Let $\PP^1$ be the complex projective line and let $\F$ be a coherent algebraic sheaf on $\PP^1 \times \PP^1$
with support of dimension $1$.
We fix the polarization $\O_{\PP^1}(1) \tensor \O_{\PP^1}(1)$ on $\PP^1 \times \PP^1$.
According to \cite[Proposition 2]{ballico_huh}, there are $r, s, t \in \ZZ$ such that for any $m, n \in \ZZ$
the Euler characteristic of the twisted sheaf $\F(m, n)$ satisfies $\chi(\F(m, n)) = rm + sn + t$.
The linear polynomial $P_{\F}(m, n) = rm + sn + t$ is called the \emph{Hilbert polynomial} of $\F$
and the ratio $\p(\F) = t / (r+s)$ is called the \emph{slope} of $\F$ with respect to the fixed polarization.
We recall that $\F$ is \emph{semi-stable} (respectively \emph{stable}) with respect to the above polarization
if it does not contain subsheaves with support of dimension zero
and for any proper subsheaf $\E \subset \F$ we have $\p(\E) \le \p(\F)$ (respectively $\p(\E) < \p(\F)$).
According to \cite{simpson}, for a given polynomial $P$,
there is a coarse moduli space, denoted $\M(P)$, that is a projective variety,
and that parametrizes S-equivalence classes of semi-stable sheaves on $\PP^1 \times \PP^1$ with Hilbert polynomial $P$.
Its dimension, as computed in \cite[Proposition 2.3]{lepotier}, is $2rs + 1$.
By the argument at \cite[Theorem 3.1]{lepotier} $\M(P)$ is irreducible
and by \cite[Proposition 2.3]{lepotier} it is smooth at the points given by stable sheaves.

The first non-trivial examples of such moduli spaces are $\M(2m+2n+1)$
and $\M(2m+2n+2)$. They were studied in \cite{ballico_huh} which contains
a classification of the semi-stable sheaves by means of locally free resolutions.
The rationality of $\M(2m+2n+2)$ was proved in \cite{chung_moon} by the wall-crossing method
and in \cite{osaka} by an elementary method.

The object of this paper is the study of $\MM = \M(3m+2n+1)$.
The points of $\MM$ are stable sheaves $\F$ supported on curves of bidegree $(2, 3)$ contained in $\PP^1 \times \PP^1$, with $\chi(\F)=1$.
As noted above, $\MM$ is a smooth projective variety of dimension $13$.
Twisting by powers of the polarization provides isomorphisms $\MM \simeq \M(3m + 2n + 5t)$
for any $t \in \ZZ$.

For $i, j \in \ZZ$ we use the abbreviation $\O(i, j) = \O_{\PP^1 \times \PP^1}(i, j)$.
We fix vector spaces $V_1$ and $V_2$ over $\CC$ of dimension $2$ and we make the identifications
\[
\PP^1 \times \PP^1 = \PP(V_1) \times \PP(V_2), \qquad
\H^0(\O(i, j)) = \operatorname{S}^i V_1^* \tensor \operatorname{S}^j V_2^*.
\]
We fix a basis $\{ x, y \}$ of $V_1^*$ and a basis $\{ z, w \}$ of $V_2^*$.
For a sheaf $\F$ we denote by $[\F]$ its S-equivalence class. If $\F$ is stable, then $[\F]$ is its isomorphism class.

\begin{theorem}
\label{main_theorem}
The variety $\MM$ is rational.
We have a decomposition of $\MM$ into an open subvariety $\MM_0$,
a closed smooth irreducible subvariety $\MM_1$ of codimension $1$,
and a closed subvariety $\MM_2 \cup \MM_3$ having two smooth irreducible components $\MM_2$, $\MM_3$
of codimension $2$, respectively, $3$. The subvarieties are defined as follows: $\MM_0 \subset \MM$ is the subset
of sheaves $\F$ having a resolution of the form
\[
0 \lra 2 \O(-1, -2) \stackrel{\f}{\lra} \O(0, -1) \oplus \O \lra \F \lra 0
\]
where $\f_{11}$ and $\f_{12}$ define a subscheme of length $2$ of $\PP^1 \times \PP^1$;
$\MM_1 \subset \MM$ is the subset of sheaves $\F$ having a resolution of the form
\[
0 \lra \O(-2, -1) \oplus \O(-1, -3) \stackrel{\f}{\lra} \O(-1, -1) \oplus \O \lra \F \lra 0
\]
where $\f_{11} \neq 0$, $\f_{12} \neq 0$;
$\MM_2$ is the set of twisted structure sheaves $\O_C(0, 1)$ for a curve $C \subset \PP^1 \times \PP^1$ of bidegree $(2, 3)$;
$\MM_3$ is the set of non-split extensions of $\O_L$ by $\O_Q$ for a line $L \subset \PP^1 \times \PP^1$ of bidegree $(0, 1)$
and a quartic $Q \subset \PP^1 \times \PP^1$ of bidegree $(2, 2)$.

The subvariety $\MM_1$ is isomorphic to a $\PP^9$-bundle over $\PP^1 \times \PP^2$ and is the Brill-Noether locus of sheaves
$\F$ satisfying $\H^0(\F(-1,1)) \neq 0$ (for $\F \in \MM_1$ we have $\H^0(\F(-1, 1)) \simeq \CC$);
$\MM_2$ is isomorphic to $\PP^{11}$ and is the Brill-Noether locus of sheaves $\F$ satisfying $\H^1(\F) \neq 0$
(for $\F \in \MM_2$ we have $\H^1(\F) \simeq \CC$);
$\MM_3$ is isomorphic to a $\PP^1$-bundle over $\PP^8 \times \PP^1$.
\end{theorem}

\noindent
The proof of this theorem is distributed among the statements of Section \ref{classification}.

As an application of our classification of sheaves we compute the Betti numbers of $\MM$.
For a projective variety $X$ we define the Poincar\'e polynomial
\[
\Poly(X)(\xi) = \sum_{i \ge 0} \dim_{\mathbb Q}^{} \H^i(X, {\mathbb Q}) \xi^{i/2}.
\]
The varieties occurring in this paper will have no odd cohomology, so the above will be a genuine polynomial expression.

\begin{theorem}
\label{poincare_polynomial}
The integral homology groups of $\MM$ have no torsion.
The Poincar\'e polynomial of $\MM$ is
\[
\xi^{13} + 3 \xi^{12} + 8 \xi^{11} + 10 \xi^{10} + 11 \xi^9 + 11 \xi^8 + 11 \xi^7 + 11 \xi^6 + 11 \xi^5 + 11 \xi^4 + 10 \xi^3 + 8 \xi^2
+ 3 \xi + 1.
\]
\end{theorem}

\noindent
The proof of this theorem takes up Section \ref{variation} and is based on the approach of Choi and Chung \cite{choi_chung},
where they study moduli spaces of $\alpha$-semi-stable pairs and their variation when the parameter $\alpha$ changes.
Thus, we show that $\MM$ is obtained from the relative Hilbert scheme of two points on the general curve of bidegree $(2, 3)$
by performing one blowing up followed by two blowing down operations.
The Betti numbers of $\MM$ have already been computed in \cite[Section 9.2]{choi_katz_klemm} in the context of physics.
Our calculation agrees with the one in \cite{choi_katz_klemm}.
The Euler characteristic of $\MM$ is $110$.

In Section \ref{cohomology} we prove that $\H^1(\F) = 0$ for $\F \in \MM \setminus \MM_2$, which is a crucial step
in our classification of sheaves. In Section \ref{preliminaries} we present our main technical tool: a spectral sequence
converging to a coherent sheaf on $\PP^1 \times \PP^1$ reminiscent to the Beilinson spectral sequence
on the projective plane.

%%%%%%%%%%%%%%%%%%%%%%%%%%%%%%%%%%%%%%%%%%%%%%%%%%%%%%%%

\section{Preliminaries}
\label{preliminaries}

According to \cite[Lemma 1]{buchdahl}, for a given coherent sheaf $\F$ on $\PP^1 \times \PP^1$ there is a spectral sequence
converging to $\F$ in degree zero and to $0$ in degrees different from zero.
The sheaves $E_1^{ij}$ on the first level $E_1$ are defined as follows:
\begin{align*}
E_1^{ij} & = 0 \quad \text{for $i > 0$ and $i < -2$,} \\
E_1^{0j} & = \H^j(\F) \tensor \O, \\
E_1^{-2,j} & = \H^j(\F(-1,-1)) \tensor \O(-1, -1).
\end{align*}
The sheaves $E_1^{-1,j}$ fit into exact sequences
\[
\H^j(\F(0, -1)) \tensor \O(0, -1) \lra E_1^{-1, j} \lra \H^j(\F(-1, 0)) \tensor \O(-1, 0).
\]
For a sheaf $\F$ with support of dimension $1$, which will be our case, the relevant part of $E_1$
is represented in the tableau
\begin{equation}
\label{E_1}
\xymatrix
{
\H^1(\F(-1, -1)) \tensor \O(-1, -1) \ar[r]^-{\f_1} & E_1^{-1, 1} \ar[r]^-{\f_2} & \H^1(\F) \tensor \O \\
\H^0(\F(-1, -1)) \tensor \O(-1, -1) \ar[r]^-{\f_3} & E_1^{-1, 0} \ar[r]^-{\f_4} & \H^0(\F) \tensor \O
}
\end{equation}
where the middle sheaves are part of the exact sequences
\begin{equation}
\label{E_1^{-1,0}}
\H^0(\F(0, -1)) \tensor \O(0, -1) \lra E_1^{-1, 0} \lra \H^0(\F(-1, 0)) \tensor \O(-1, 0),
\end{equation}
\begin{equation}
\label{E_1^{-1,1}}
\H^1(\F(0, -1)) \tensor \O(0, -1) \lra E_1^{-1, 1} \lra \H^1(\F(-1, 0)) \tensor \O(-1, 0).
\end{equation}
The relevant part of the second level of the spectral sequence is represented in the tableau
\[
\xymatrix
{
\Ker(\f_1) \ar[rrd]^-{\f_5} & \Ker(\f_2)/\Image(\f_1) & \Coker(\f_2) \\
\Ker(\f_3) & \Ker(\f_4)/\Image(\f_3) & \Coker(\f_4)
}
\]
The spectral sequence degenerates at $E_3 = E_{\infty}$. The convergence of the spectral sequence implies
that $\f_2$ is surjective and that we have the exact sequence
\begin{equation}
\label{convergence}
0 \lra \Ker(\f_1) \stackrel{\f_5}{\lra} \Coker(\f_4) \lra \F \lra \Ker(\f_2)/\Image(\f_1) \lra 0.
\end{equation}

Let $\E$ be a semi-stable sheaf on $\PP^1 \times \PP^1$ with $P_{\E}(m,n) = rm + n + 1$.
According to \cite[Proposition 11]{ballico_huh}, $\E$ has resolution
\begin{equation}
\label{(r,1)}
0 \lra \O(-1, -r) \lra \O \lra \E \lra 0.
\end{equation}
Let $\E$ be a semi-stable sheaf on $\PP^1 \times \PP^1$ with $P_{\E}(m,n) = m + sn + 1$.
Then $\E$ has resolution
\begin{equation}
\label{(1,s)}
0 \lra \O(-s, -1) \lra \O \lra \E \lra 0.
\end{equation}
According to \cite[Proposition 14]{ballico_huh}, a semi-stable sheaf $\E$ on $\PP^1 \times \PP^1$ with Hilbert polynomial
$2m + 2n + 1$ has resolution
\begin{equation}
\label{(2,2)}
0 \lra \O(-2, -1) \oplus \O(-1, -2) \lra \O(-1, -1) \oplus \O \lra \E \lra 0.
\end{equation}

For a sheaf $\F$ of dimension $1$, without zero-dimensional torsion, on $\PP^1 \times \PP^1$ we define the \emph{dual sheaf}
\[
\F^\dual = {\mathcal Ext}^1_{\O_{\PP^1 \times \PP^1}} (\F, \omega_{\PP^1 \times \PP^1}).
\]

\begin{lemma}
\label{duality}
The map $[\F] \mapsto [\F^\dual]$ is well-defined and gives an isomorphism
\[
\M(rm + sn + t) \lra \M(rm + sn - t).
\]
\end{lemma}

\begin{proof}
Consider the Segre embedding $\PP^1 \times \PP^1 \subset \PP^3$. Then the dual of $\F$ as a sheaf on $\PP^3$
is compatible with the dual of $\F$ as a sheaf on $\PP^1 \times \PP^1$:
\[
\F^\dual \simeq {\mathcal Ext}^{2}_{\O_{\PP^3}}(\F, \omega_{\PP^3})|_{\PP^1 \times \PP^1}.
\]
This allows us to apply \cite[Theorem 13]{rendiconti} to obtain the conclusion.
\end{proof}

\noindent
In particular, $\MM \simeq \M(3m + 2n - 1)$.
Note that the same argument applies for moduli spaces of one-dimensional sheaves on smooth projective varieties.

\begin{theorem}
\label{M_dual}
We have a decomposition of $\M(3m + 2n - 1)$ into subsets $\MM_0^\dual$, $\MM_1^\dual$,
$\MM_2^\dual \cup \MM_3^\dual$, where $\MM_i^\dual$ is the image of $\MM_i$ under the above isomorphism.
Thus, $\MM_0^\dual$ is the subset of sheaves $\F$ having a resolution of the form
\[
0 \lra \O(-2, -2) \oplus \O(-2, -1) \stackrel{\f}{\lra} 2\O(-1, 0) \lra \F \lra 0,
\]
where $\f_{12}$ and $\f_{22}$ define a zero-dimensional subscheme of $\PP^1 \times \PP^1$;
$\MM_1^\dual$ is the subset of sheaves $\F$ having a resolution of the form
\[
0 \lra \O(-2, -2) \oplus \O(-1, -1) \stackrel{\f}{\lra} \O(-1, 1) \oplus \O(0, -1) \lra \F \lra 0,
\]
where $\f_{12} \neq 0$, $\f_{22} \neq 0$;
$\MM_2^\dual$ is the set of structure sheaves of curves of bidegree $(2, 3)$;
$\MM_3^\dual$ is the set of non-split extensions
of $\O_Q$ by $\O_L(-2, -1)$ with $L$ a line of bidegree $(0, 1)$ and $Q$ a quartic of bidegree $(2, 2)$.
\end{theorem}

%%%%%%%%%%%%%%%%%%%%%%%%%%%%%%%%%%%%%%%%%%%%%%%%%%%%

\section{Vanishing of cohomology}
\label{cohomology}

\noindent
The following lemma is analogous to \cite[Lemma 6.7]{pacific}.
We will use the word \emph{curve} to denote a subscheme defined by a polynomial equation.

\begin{lemma}
\label{pacific}
Let $C \subset \PP^1 \times \PP^1$ be a curve and $\I \subset \O_C$ an ideal sheaf.
Then there is a curve $C' \subset C$ such that the ideal sheaf of $C'$ in $C$, denoted $\I'$, contains $\I$,
and $\I'/\I$ has support of dimension at most $0$.
\end{lemma}

\noindent
The following proposition is a strengthtening of \cite[Lemma 9]{ballico_huh}.

\begin{proposition}
\label{O_C_stable}
Let $C \subset \PP^1 \times \PP^1$ be a curve of bidegree $(s, r)$. Then $\O_C$ is semi-stable.
If $r > 0$ and $s > 0$, then $\O_C$ is stable.
\end{proposition}

\begin{proof}
Let $\I \subset \O_C$ be a proper subsheaf and let $\I'$ and $C'$ be as in Lemma \ref{pacific}.
Let $t$ be the length of $\I'/\I$ and let $(s', r')$ be the bidegree of $C'$. The Hilbert polynomial of $\I$ is given by
\begin{align*}
P_{\I} & = P_{\I'} - t = P_{\O_C} - P_{\O_{C'}} - t \\
& = rm + sn + r + s -rs - r'm - s'n - r' - s' + r's' - t.
\end{align*}
Thus, the slopes of $\I$ and $\O_C$ are given by
\[
\p(\I) = 1 + \frac{r's' - rs - t}{r - r' + s - s'}, \qquad \p(\O_C) = 1 - \frac{rs}{r + s}.
\]
The inequality $\p(\I) \le \p(\O_C)$ follows from the inequality $0 \le rr'(s-s') + ss'(r-r')$.
If $r > 0$ and $s > 0$, then this inequality is strict because either $r' < r$ or $s' < s$.
\end{proof}

\begin{proposition}
\label{vanishing}
Let $\F$ be a semi-stable sheaf on $\PP^1 \times \PP^1$ with Hilbert polynomial $P_{\F}(m, n) = rm + sn + t$.
Let $i$ and $j$ be integers.
\begin{enumerate}
\item[(i)] If ${\displaystyle \max \{ i, j \} < 1 - \frac{rs + t}{r + s}}$, then $\H^0(\F(i, j)) = 0$.
\item[(ii)] If ${\displaystyle \min \{ i, j \} > -1 + \frac{rs - t}{r + s}}$, then $\H^1(\F(i, j)) = 0$.
\end{enumerate}
\end{proposition}

\begin{proof}
Assume that $\H^0(\F(i, j)) \neq 0$. Then there is a non-zero morphism $\alpha \colon \O_D \to \F(i, j)$
for a curve $D \subset \PP^1 \times \PP^1$. Let $\J = \Ker(\alpha)$. By Lemma \ref{pacific} there is a curve $C \subset D$
such that the ideal sheaf $\I$ of $C$ in $\O_D$ contains $\J$ and $\I/\J$ is supported on finitely many points.
Since $\F(i, j)$ has no zero-dimensional torsion, $\alpha(\I/\J) = 0$, hence $\J = \I$,
and hence $\alpha$ factors through an injective morphism $\O_C \to \F(i, j)$.
From the semi-stability of $\F$ we get the inequality
\[
\p(\O_C(-i, -j)) = 1 - \frac{r's' + r'i + s'j}{r' + s'} \le \frac{t}{r + s} = \p(\F).
\]
Combining this with the inequalities
\[
- \frac{rs}{r + s} \le - \frac{r's'}{r' + s'}, \qquad - \max \{ i, j \} \le - \frac{r'i + s'j}{r' + s'}
\]
we obtain the inequality
\[
1 - \frac{rs}{r + s} - \max \{ i, j \} \le \frac{t}{r + s}.
\]
This contradicts the hypothesis of (i). Part (ii) follows from (i) and Serre duality.
We have
\[
\H^1(\F(i, j)) \simeq (\H^0(\F^\dual(-i,-j)))^*
\]
and, by Lemma \ref{duality}, $\F^\dual$ is semi-stable with Hilbert polynomial $rm + sn - t$.
Thus, the right-hand-side vanishes if
\[
\max \{ -i, -j \} < 1 - \frac{rs - t}{r + s}. \qedhere
\]
\end{proof}

\noindent
Using this proposition we can give another proof to the fact shown at \cite[Proposition 10]{ballico_huh}
that there are no semi-stable sheaves on $\PP^1 \times \PP^1$ with Hilbert polynomial $r m + t$,
for $r \ge 2$ and $t$ not a multiple of $r$.

\begin{corollary}
\label{empty_spaces}
The moduli spaces $\M(rm + t)$ are empty for $r \ge 2$ and $0 < t < r$.
\end{corollary}

\begin{proof}
Assume that $\F$ is a semi-stable sheaf in one of these moduli spaces.
From Proposition \ref{vanishing} (i) we get $\H^0(\F) = 0$.
From Proposition \ref{vanishing} (ii) we get $\H^1(\F) = 0$.
Thus, $t = \chi(\F) = 0$, which contradicts our choice of $t$.
\end{proof} 

\begin{proposition}
\label{vanishing_M}
For $\F \in \MM$ we have $\H^0(\F(-1, -1)) = 0$, $\H^0(\F(-1, 0)) = 0$, and $\H^0(\F(0, -1)) \neq 0$ if and only if
$\F \simeq \O_C(0, 1)$ for a curve $C \subset \PP^1 \times \PP^1$ of bidegree $(2, 3)$.
\end{proposition}

\begin{proof}
The vanishing of $\H^1(\F(-1, -1))$ follows from Proposition \ref{vanishing} (i).
Assume that $\H^0(\F(i, j)) \neq 0$, where $(i, j) = (-1, 0)$ or $(0, -1)$.
As in the proof of Proposition \ref{vanishing}, there is a curve $C$ and an injective morphism $\O_C \to \F(i, j)$.
In Table 1 below we have the possible bidegrees of $C$ and the slopes of $\O_C(-i, -j)$.

\begin{table}[ht]{Table 1. Possibilities for $C$.}
\begin{center}
\begin{tabular}{| c | c | c | c |}
\hline
$\deg(C)$ & $P_{\O_C}$ & $\p(\O_C(1, 0))$ & $\p(\O_C(0, 1))$
\\
\hline
$(2, 3)$ & $3m + 2n - 1$ & $2/5$ & $1/5$
\\
$(2, 2)$ & $2m + 2n$ & $1/2$ & $1/2$
\\
$(1, 3)$ & $3m + n + 1$ & $1$ & $1/2$
\\
$(2, 1)$ & $m + 2n + 1$ & $2/3$ & $1$
\\
$(1, 2)$ & $2m + n + 1$ & $1$ & $2/3$
\\
$(0, 3)$ & $3m + 3$ & $2$ & $1$
\\
$(2, 0)$ & $2n + 2$ & $1$ & $2$
\\
$(1, 1)$ & $m + n + 1$ & $1$ & $1$
\\
$(0, 2)$ & $2m + 2$ & $2$ & $1$
\\
$(1, 0)$ & $n + 1$ & $1$ & $2$
\\
$(0, 1)$ & $m + 1$ & $2$ & $1$
\\
\hline
\end{tabular}
\end{center}
\end{table}

The only case in which $\O_C(-i, -j)$ does not violate the semi-stability of $\F$ is when $\deg(C) = (2, 3)$ and $(i, j) = (0, -1)$.
We deduce that $\H^0(\F(-1, 0)) = 0$ and, if $\H^0(\F(0, -1)) \neq 0$,
then $\F \simeq \O_C(0, 1)$ for a curve $C \subset \PP^1 \times \PP^1$ of bidegree $(2, 3)$.
It remains to show that $\O_C(0, 1)$ is semi-stable.
Let $\I \subset \O_C$ be an ideal sheaf and let $\I'$ and $C'$ be as in Lemma \ref{pacific}.
In Table 2 below we have the possible bidegrees of $C'$ and the resulting slopes of $\I'(0, 1)$.

\begin{table}[ht]{Table 2. Possibilities for $C'$.}
\begin{center}
\begin{tabular}{| c | c | c | c |}
\hline
$\deg(C')$ & $P_{\O_{C'}}$ & $P_{\I'}$ & $\p(\I'(0, 1))$
\\
\hline
$(2, 2)$ & $2m + 2n$ & $m - 1$ & $-1$
\\
$(1, 3)$ & $3m + n + 1$ & $n - 2$ & $-1$
\\
$(2, 1)$ & $m + 2n + 1$ & $2m - 2$ & $-1$
\\
$(1, 2)$ & $2m + n + 1$ & $m + n - 2$ & $-1/2$
\\
$(0, 3)$ & $3m + 3$ & $2n - 4$ & $-1$
\\
$(2, 0)$ & $2n + 2$ & $3m - 3$ & $-1$
\\
$(1, 1)$ & $m + n + 1$ & $2m + n - 2$ & $-1/3$
\\
$(0, 2)$ & $2m + 2$ & $m + 2n - 3$ & $-1/3$
\\
$(1, 0)$ & $n + 1$ & $3m + n - 2$ & $-1/4$
\\
$(0, 1)$ & $m + 1$ & $2m + 2n - 2$ & $0$
\\
\hline
\end{tabular}
\end{center}
\end{table}

\noindent
In all cases $\p(\I'(0, 1)) < \p(\O_C(0, 1))$. In conclusion, $\O_C(0, 1)$ is semi-stable.
\end{proof}

\begin{proposition}
\label{H^1_vanishing}
Let $\F$ give a point in $\MM$. If $\H^0(\F(0, -1)) = 0$, then $\H^1(\F) = 0$.
\end{proposition}

\begin{proof}
Denote $d = \dim \H^1(\F)$. In view of Proposition \ref{vanishing_M} and exact sequence (\ref{E_1^{-1,0}})
we get $E_1^{-1,0} = 0$. Thus, the exact sequence (\ref{convergence}) becomes
\[
0 \lra \Ker(\f_1) \stackrel{\f_5}{\lra} (d+1)\O \lra \F \lra \Ker(\f_2)/\Image(\f_1) \lra 0.
\]
From Proposition \ref{vanishing_M} we get $\H^1(\F(-1, -1)) \simeq \CC^4$, hence we have the exact sequence
\[
0 \lra \Ker(\f_1) \lra 4\O(-1, -1) \lra \Image(\f_1) \lra 0.
\]
We also have the exact sequence
\[
0 \lra \Ker(\f_2) \lra E_1^{-1,1} \lra d \O \lra 0.
\]
From these exact sequences we can compute the Hilbert polynomial of $E_1^{-1,1}$:
\begin{align*}
P_{E_1^{-1,1}} & = P_{\Ker(\f_2)} + d P_{\O} = P_{\Ker(\f_2)/\Image(\f_1)} + P_{\Image(\f_1)} + d P_{\O} \\
& = P_{\F} - (d+1) P_{\O} + P_{\Ker(\f_1)} + 4 P_{\O(-1, -1)} - P_{\Ker(\f_1)} + d P_{\O} \\
& = P_{\F} - P_{\O} + 4 P_{\O(-1, -1)} \\
& = 3mn + 2m + n.
\end{align*}
The exact sequence (\ref{E_1^{-1,1}}) becomes
\[
\O(0, -1) \lra E_1^{-1, 1} \lra 2\O(-1,0).
\]
Since $P_{E_1^{-1,1}} = P_{\O(0, -1)} + 2 P_{\O(-1, 0)}$ this sequence is also exact on the left and right, and, in fact,
it is split exact. We deduce that $E_1^{-1,1} \simeq \O(0, -1) \oplus 2\O(-1, 0)$.
It follows that $d \le 2$ because there is, obviously, no surjective morphism
\[
\f_2 \colon \O(0, -1) \oplus 2\O(-1, 0) \lra d \O
\]
for $d \ge 3$. Assume that $d = 2$. Then the maximal minors of $\f_2$ have no common factor, otherwise $\f_2$
would not be surjective. It follows that $\Ker(\f_2) \simeq \O(-2, -1)$. We have
\[
P_{\Coker(\f_2)} = 2 P_{\O} - P_{\O(0, -1)} - 2 P_{\O(-1, 0)} + P_{\O(-2, -1)} = 2
\]
which contradicts the surjectivity of $\f_2$.
Assume that $d = 1$. If the restriction of $\f_2$ to $\O(0, -1)$ were zero, then $\Ker(\f_2) \simeq \O(0, -1) \oplus \O(-2, 0)$.
This would yield a contradiction because $\Ker(\f_2)/\Image(\f_1)$ would contain $\O(-2, 0)$ as a direct summand,
but there is no surjective morphism $\F \to \O(-2, 0)$. Thus, we may write
\[
\f_2 = \left[
\begin{array}{ccc}
-1 \tensor z & x \tensor 1 & y \tensor 1
\end{array}
\right],
\]
\[
\f_1 = \left[
\begin{array}{cccc}
x \tensor 1 & y \tensor 1 & 0 & 0 \\
1 \tensor z & 0 & 0 & 0 \\
0 & 1 \tensor z & 0 & 0
\end{array}
\right].
\]
It follows that $\Ker(\f_1) \simeq 2\O(-1, -1)$.
Thus, $\Coker(\f_5)$ has Hilbert polynomial $2 P_{\O} - 2 P_{\O(-1, -1)} = 2m + 2n + 2$, hence it has slope $1/2$,
and hence it is a destabilizing subsheaf of $\F$.
In conclusion, $d = 0$.
\end{proof}

%%%%%%%%%%%%%%%%%%%%%%%%%%%%%%%%%%%%%%%%%%%%%%%%%%%%%

\section{Classification of sheaves}
\label{classification}

Assume that $\F$ gives a point in $\MM$ and that $\H^0(\F(0, -1)) = 0$.
Then, as seen at Proposition \ref{H^1_vanishing}, $\H^1(\F) = 0$, and, as seen in the proof of this proposition,
$E_1^{-1, 1} \simeq \O(0, -1) \oplus 2\O(-1, 0)$. Thus, the exact sequence (\ref{convergence}) becomes
\begin{equation}
\label{generic_convergence}
0 \lra \Ker(\f_1) \stackrel{\f_5}{\lra} \O \lra \F \lra \Coker(\f_1) \lra 0,
\end{equation}
where
\[
\f_1 \colon 4 \O(-1, -1) \lra \O(0, -1) \oplus 2\O(-1, 0).
\]

\begin{lemma}
\label{generic_sheaves}
Assume that $\F$ gives a point in $\MM$ and that $\H^0(\F(0, -1)) = 0$.
Assume that the maximal minors of $\f_1$ have no common factor.
Then $\Ker(\f_1) \simeq \O(-2, -3)$ and $\Coker(\f_1)$ is isomorphic to the structure sheaf of a zero-dimensional
subscheme $Z \subset \PP^1 \times \PP^1$ of length $2$.
Moreover, $Z$ is not contained in a line of bidegree $(0, 1)$.
Thus, we have a non-split extension
\begin{equation}
\label{generic_extension}
0 \lra \O_C \lra \F \lra \O_Z \lra 0
\end{equation}
where $C$ is a curve of bidegree $(2, 3)$ containing $Z$.
\end{lemma}

\begin{proof}
Let $\zeta_j$, $1 \le j \le 4$, be the maximal minor of $\f_1$ obtained by deleting column $j$, for a matrix representation
of $\f_1$. It is well-known that the sequence
\[
0 \lra \O(-2, -3) \stackrel{\zeta}{\lra} 4\O(-1, -1) \stackrel{\f_1}{\lra} \O(0, -1) \oplus 2\O(-1, 0),
\]
\[
\zeta = \left[
\begin{array}{cccc}
\zeta_1 & -\zeta_2 & \phantom{-} \zeta_3 & - \zeta_4
\end{array}
\right]^\trans,
\]
is exact. Let $Z \subset \PP^1 \times \PP^1$ be the subscheme given by the ideal $(\zeta_1, \zeta_2, \zeta_3, \zeta_4)$.
The Hilbert polynomial of $\O_Z$ can be computed from the exact sequence
\[
0 \lra \O(-2, -2) \oplus 2\O(-1, -3) \stackrel{\f_1^\trans}{\lra} 4\O(-1, -2) \stackrel{\zeta^\trans}{\lra} \O \lra \O_Z \lra 0.
\]
We get $P_{\O_Z} = 2$, hence $Z$ is zero-dimensional of length $2$.
From the short exact sequence
\[
0 \lra \I_Z \lra \O \lra \O_Z \lra 0
\]
we get the long exact sequence
\begin{align*}
0 & \lra {\mathcal Hom}(\O_Z, \O) \lra {\mathcal Hom}(\O, \O) \lra {\mathcal Hom}(\I_Z, \O) \\
& \lra {\mathcal Ext}^1(\O_Z, \O) \lra {\mathcal Ext}^1(\O, \O) \lra {\mathcal Ext}^1(\I_Z, \O) \\
& \lra {\mathcal Ext}^2(\O_Z, \O) \lra {\mathcal Ext}^2(\O, \O).
\end{align*}
The sheaves ${\mathcal Hom}(\O_Z, \O)$, ${\mathcal Ext}^1(\O_Z, \O)$, ${\mathcal Ext}^1(\O, \O)$, ${\mathcal Ext}^2(\O, \O)$
are zero, hence we get the isomorphisms
\[
{\mathcal Hom}(\I_Z, \O) \simeq \O, \qquad {\mathcal Ext}^1(\I_Z, \O) \simeq {\mathcal Ext}^2(\O_Z, \O) \simeq \O_Z.
\]
We apply the long ${\mathcal Ext}(-,\O)$-sequence to the short exact sequence
\[
0 \lra \O(-2, -2) \oplus 2\O(-1, -3) \stackrel{\f_1^\trans}{\lra} 4\O(-1, -2) \lra \I_Z \lra 0
\]
and we use the above isomorphisms to obtain the exact sequence
\[
0 \lra \O \lra 4\O(1, 2) \stackrel{\psi}{\lra} \O(2, 2) \oplus 2\O(1, 3) \lra \O_Z \lra 0.
\]
The morphism $\psi$ is a twist of $\f_1$.
Assume that $Z$ were contained in a line of bidegree $(0, 1)$. Then we would have a commutative diagram
\[
\xymatrix
{
0 \ar[r] & \O(-2, -2) \oplus 2\O(-1, -3) \ar[r] \ar[d]^-{\beta} & 4\O(-1, -2) \ar[r] \ar[d]^-{\alpha} & \I_Z \ar[r] \ar@{=}[d] & 0 \\
0 \ar[r] & \O(-2, -1) \ar[r] & \O(-2, 0) \oplus \O(0, -1) \ar[r] & \I_Z \ar[r] & 0
}
\]
in which $\alpha \neq 0$. Thus $\rank (\Ker(\alpha)) = 3$, hence $\beta = 0$, and hence $\Coker(\beta) \simeq \O(-2, -1)$
contains $\O(-2, 0)$ as a direct summand. This is absurd.

The exact sequence (\ref{generic_extension}) follows from (\ref{generic_convergence}) with $\O_C = \Coker(\f_5)$.
From sequence (\ref{generic_extension}), and since $\F$ has no zero-dimensional torsion,
we see that $\F$ has schematic support $C$, hence $Z$ is contained in $C$.
\end{proof}

\begin{lemma}
\label{unique_extension}
Let $C \subset \PP^1 \times \PP^1$ be a curve of bidegree $(2, 3)$ and let $Z \subset C$ be a zero-dimensional subscheme of length $2$.
Let $\F$ be an extension of $\O_Z$ by $\O_C$ that has no zero-dimensional torsion.
Then $\F$ is uniquely determined up to isomorphism.
This means that if $\G$ is another extension of $\O_Z$ by $\O_C$ that has no zero-dimensional torsion, then $\F \simeq \G$.
\end{lemma}

\begin{proof}
By Serre duality $\Ext^1(\O_Z, \O_C) \simeq (\Ext^1(\O_C, \O_Z))^*$.
From the short exact sequence
\begin{equation}
\label{O_C_resolution}
0 \lra \O(-2, -3) \lra \O \lra \O_C \lra 0
\end{equation}
we get the long exact sequence
\begin{multline*}
0 \lra \Hom(\O_C, \O_Z) \simeq \H^0(\O_Z) \simeq \CC^2 \lra \Hom(\O, \O_Z) \simeq \CC^2 \\
\lra \Hom(\O(-2, -3), \O_Z) \simeq \CC^2 \lra \Ext^1(\O_C, \O_Z) \lra \Ext^1(\O, \O_Z) = 0
\end{multline*}
We obtain $\Ext^1(\O_Z, \O_C) \simeq \CC^2$.

Assume that $Z = \{ p, q \}$ for distinct points $p, q \in C$.
We denote by $\CC_p$ and $\CC_q$ the structure sheaves of the subschemes $\{ p \}$, respectively, $\{ q \} \subset \PP^1 \times \PP^1$.
From sequence (\ref{O_C_resolution}) we get the long exact sequence
\begin{multline*}
0 \lra \Hom (\O_C, \CC_p) \simeq \CC \lra \Hom(\O, \CC_p) \simeq \CC \lra \Hom(\O(-2, -3), \CC_p) \simeq \CC \\
\lra \Ext^1(\O_C, \CC_p) \simeq (\Ext^1(\CC_p, \O_C))^* \lra \Ext^1(\O, \CC_p) = 0.
\end{multline*}
Thus, there is a unique non-trivial extension of $\CC_p$ by $\O_C$, denoted by $\E$.
From the short exact sequence
\[
0 \lra \O_C \lra \E \lra \CC_p \lra 0
\]
we get the long exact sequence
\[
0 = \Hom(\CC_q, \CC_p) \lra \Ext^1(\CC_q, \O_C) \simeq \CC \lra \Ext^1(\CC_q, \E) \lra \Ext^1(\CC_q, \CC_p) = 0.
\]
Thus, there is a unique non-trivial extension of $\CC_q$ by $\E$, hence $\F$ is unique up to isomorphism.

We next consider the case when $Z$ is a double point supported on $p \in C$.
We construct a resolution of $\E$ by combining resolution (\ref{O_C_resolution}) with the resolution
\[
0 \lra \O(-2, -3) \lra \O(-2, -2) \oplus \O(-1, -3) \lra \O(-1, -2) \lra \CC_p \lra 0.
\]
The map $\O(-1, -2) \to \CC_p$ lifts to $\E$ because $\H^1(\O_C(1, 2)) = 0$.
Applying the argument at the proof of \cite[Proposition 2.3.2]{illinois}, which uses the fact that $\Ext^1(\CC_p, \O) = 0$,
we can show that the induced map $\O(-2, -3) \to \O(-2, -3)$ is non-zero.
We obtain the resolution
\[
0 \lra \O(-2, -2) \oplus \O(-1, -3) \stackrel{\psi}{\lra} \O(-1, -2) \oplus \O \lra \E \lra 0
\]
in which $\psi_{11} \neq 0$, $\psi_{12} \neq 0$, $\psi_{11}(p) = 0$, $\psi_{12}(p) = 0$.
Moreover, $\psi_{21}(p) = 0$ and $\psi_{22}(p) = 0$ if and only if $p$ is a singular point of $C$.
We have the exact sequence
\begin{multline*}
\Hom(\O(-1, -2) \oplus \O, \CC_p) \simeq \CC^2 \xrightarrow{\psi(p)} \Hom(\O(-2, -2) \oplus \O(-1, -3), \CC_p) \simeq \CC^2 \\
\lra \Ext^1(\E, \CC_p) \simeq (\Ext^1(\CC_p, \E))^* \lra \Ext^1(\O(-1, -2) \oplus \O, \CC_p) = 0.
\end{multline*}
We get a unique non-trivial extension of $\CC_p$ by $\E$ if $p$ is a regular point of $C$.
In this case $\F$ is unique up to isomorphism.

Assume now that $p$ is a singular point of $C$. Then $\psi(p) = 0$, hence $\Ext^1(\CC_p, \E) \simeq \CC^2$.
According to \cite[Proposition 2.3.1]{huybrechts_lehn}, the subset $U_Z \subset \PP(\Ext^1(\O_Z, \O_C)) \simeq \PP^1$
of extension sheaves having no zero-dimensional torsion is open.
We construct a map $\upsilon_Z \colon U_Z \to \PP(\Ext^1(\CC_p, \E)) \simeq \PP^1$ as follows.
Let $\I$ be the ideal sheaf of $\{ p \}$ in $Z$. Note that $\I \simeq \CC_p$ as modules over $\O$.
Given $\F \in U_Z$ let $\A$ be the pull-back in $\F$ of $\I$.
Then there is a unique isomorphism $\E \to \A$ making the diagram commute
\[
\xymatrix
{
0 \ar[r] & \O_C \ar[r] \ar@{=}[d] & \E \ar[r] \ar[d] & \CC_p \ar[r] \ar@{=}[d] & 0 \\
0 \ar[r] & \O_C \ar[r] & \A \ar[r] & \CC_p \ar[r] & 0
}
\]
The composite map $\E \to \A \to \F$ has cokernel $\CC_p$, so $\F$ is an extension of $\CC_p$ by $\E$.
We claim that the image of $\upsilon_Z$ is a point.
If we can prove this claim, then it will follow that $\F$ is uniquely determined up to isomorphism.
Assume that the image of $\upsilon_Z$ is an open subset of $\PP^1$.
The zero-dimensional schemes $Z'$ of length $2$ supported on $p$ are parametrized by $\PP^1$.
Thus there is $Z' \neq Z$ such that $\upsilon_{Z'}(U_{Z'}) \cap \upsilon_Z (U_Z) \neq \emptyset$.
This means that we have extensions $\F \in U_Z$, $\F' \in U_{Z'}$, and a commutative diagram with exact rows
\[
\xymatrix
{
0 \ar[r] & \E \ar@{=}[d] \ar[r] & \F \ar[r] \ar[d]^-{\simeq} & \CC_p \ar[r] \ar@{=}[d] & 0 \\
0 \ar[r] & \E \ar[r] & \F' \ar[r] & \CC_p \ar[r] & 0
}
\]
The isomorphism $\F \to \F'$ fits into a commutative square
\[
\xymatrix
{
\O_C \ar[r] \ar@{=}[d] & \F \ar[d]^-{\simeq} \\
\O_C \ar[r] & \F'
}
\]
We get an induced isomorphism of cokernels $\O_Z \to \O_{Z'}$, which contradicts our choice of $Z'$.
In conclusion, the image of $\upsilon_Z$ is a point.
\end{proof}

\noindent
The difficult case in the previous lemma is when $Z$ is concentrated in one point.
For this case we will give an alternate more general argument.
The following lemma and its proof were provided by Jean-Marc Dr\'ezet, to whom the author is grateful.

\begin{lemma}
\label{general_extensions}
Let $S$ be a smooth projective surface and $C \subset S$ a Cohen-Macaulay curve.
Let $Z \subset C$ be a zero-dimensional subscheme of length $2$ concentrated on a single point $p$,
and $\L$ a line bundle on $C$.
Then there exists an extension
\begin{equation}
\label{general_extension}
0 \lra \L \lra \F \lra \O_Z \lra 0
\end{equation}
where $\F$ has no zero-dimensional torsion. The sheaf $\F$ is unique up to isomorphism.
\end{lemma}

\begin{proof}
The extensions (\ref{general_extension}) on $C$ and on $S$ are the same. Indeed, by \cite[Proposition 2.2.1]{drezet_multiples}
we have the exact sequence
\[
0 \lra \Ext^1_{\O_C}(\O_Z, \L) \lra \Ext^1_{\O_S}(\O_Z, \L) \lra \Hom({\mathcal Tor}_1^{\O_S}(\O_Z, \O_C), \L).
\]
The group on the right vanishes because ${\mathcal Tor}_1^{\O_S}(\O_Z, \O_C)$ is supported on $Z$,
yet $\L$ has no zero-dimensional torsion.
From Serre duality we have
\begin{equation}
\label{general_duality}
\Ext^1_{\O_S}(\O_Z, \L) \simeq \Ext^1_{\O_S}(\L, \O_Z \tensor \omega_S)^*.
\end{equation}
Again from \cite[Proposition 2.2.1]{drezet_multiples} we have the exact sequence
\begin{multline*}
\Ext^1_{\O_C}(\L, \O_Z \tensor \omega_S) \lra \Ext^1_{\O_S}(\L, \O_Z \tensor \omega_S) \lra
\Hom({\mathcal Tor}_1^{\O_S}(\L, \O_C), \O_Z \tensor \omega_S) \\
\lra \Ext^2_{\O_C}(\L, \O_Z \tensor \omega_S).
\end{multline*}
The first and the last groups vanish, hence we obtain the functorial isomorphisms
\begin{align*}
\Ext^1_{\O_S}(\L, \O_Z \tensor \omega_S) & \simeq \Hom({\mathcal Tor}_1^{\O_S}(\L, \O_C), \O_Z \tensor \omega_S) \\
& \simeq \Hom(\L(-C), \O_Z \tensor \omega_S) \\
& \simeq \H^0((\L^*(C) \tensor \omega_S)|_{Z}) \simeq \CC^2.
\end{align*}
Now consider an extension (\ref{general_extension}) which is non-split,
and suppose that $\F$ has a zero-dimensional subsheaf ${\mathcal T}$.
Since $\L$ is torsion-free on $C$ the composition ${\mathcal T} \to \F \to \O_Z$ is injective.
There are only two non-zero subsheaves of $\O_Z$: the sheaf of sections vanishing at $p$, which is isomorphic to $\CC_p$, and $\O_Z$
itself. Since the extension is non-split, we have ${\mathcal T} = \CC_p$.
Let $\G = \F/{\mathcal T}$. We have a commutative diagram with exact rows and columns
\[
\xymatrix
{
& & 0 \ar[d] & 0 \ar[d] \\
& & \CC_p \ar@{=}[r] \ar[d] & \CC_p \ar[d] \\
0 \ar[r] & \L \ar@{=}[d] \ar[r] & \F \ar[r] \ar[d] & \O_Z \ar[r] \ar[d] & 0 \\
0 \ar[r] & \L \ar[r] & \G \ar[r] \ar[d] & \CC_p \ar[r] \ar[d] & 0 \\
& & 0 & 0
}
\]
Let $\sigma \in \Ext^1_{\O_C}(\O_Z, \L)$ correspond to extension (\ref{general_extension}) and let $\tau \in \Ext^1_{\O_C}(\CC_p, \L)$
correspond to the extension
\[
0 \lra \L \lra \G \lra \CC_p \lra 0.
\]
Consider the morphism
\[
\Phi \colon \Ext^1_{\O_C}(\CC_p, \L) \lra \Ext^1_{\O_C}(\O_Z, \L)
\]
induced by the surjective morphism $\O_Z \to \CC_p$.
It is then easy to see that $\Phi(\tau) = \sigma$ (see \cite[Proposition 4.3.1]{drezet_deformations}).
It follows that for an extension (\ref{general_extension}) associated to $\sigma$,
the sheaf $\F$ has zero-dimensional torsion if and only if $\sigma \in \operatorname{Im}(\Phi)$.

According to (\ref{general_duality}) and to the above functorial isomorphisms, $\Phi$ is the transpose of the canonical surjective
morphism
\[
\xymatrix
{
\Psi \colon \H^0((\L^*(C) \tensor \omega_S)|_{Z}) \ar[r] \ar@{=}[d] & \H^0((\L^*(C) \tensor \omega_S)|_{p}) \ar@{=}[d] \\
\CC^2 & \CC
}
\]
The kernel of $\Psi$ is the set $m_p \simeq \CC$ of sections vanishing at $p$.
Then $\sigma \in \operatorname{Im}(\Phi)$ if and only if $\sigma$ vanishes on $m_p$.
The set of extensions $\sigma$ that do not vanish on $m_p$ is non-empty.
This proves the existence part of the lemma.
It is easy to check that the group of automorphisms of $\O_Z$ acts transitively on the set of extensions $\sigma$ that do not vanish
on $m_p$. This proves the uniqueness part of the lemma.
\end{proof}

\begin{proposition}
\label{sheaves_M_0}
Let $\F$ be an extension as in (\ref{generic_extension}), without zero-dimensional torsion,
for a curve $C$ of bidegree $(2, 3)$ and a subscheme $Z \subset C$
that is the intersection of two curves of bidegree $(1, 1)$. Then $\F$ gives a point in $\MM$.
Let $\MM_0 \subset \MM$ be the subset of such sheaves $\F$.
Then $\MM_0$ is open and it can be described as the set of sheaves $\G$ having a resolution of the form
\begin{equation}
\label{M_0}
0 \lra 2\O(-1, -2) \stackrel{\f}{\lra} \O(0, -1) \oplus \O \lra \G \lra 0
\end{equation}
where $\f_{11}$ and $\f_{12}$ define a zero-dimensional subscheme of $\PP^1 \times \PP^1$.
\end{proposition}

\begin{proof}
Any $\Coker(\f)$ is an extension of $\O_Z$ by $\O_C$ without zero dimensional torsion,
where $Z = \{ \f_{11} = 0,\, \f_{12} = 0 \}$ and $C = \{ \det \f = 0 \}$,
hence it is the unique extension of $\O_Z$ by $\O_C$ that has no zero-dimensional torsion.
It remains to show that any sheaf $\G$ having resolution (\ref{M_0}) is semi-stable.
Assume that $\G$ had a destabilizing subsheaf $\E$. Without loss of generality we may take $\E$ to be semi-stable.
Since $\dim \H^0(\G) = 1$, we have $\chi(\E) = 1$.
According to Corollary \ref{empty_spaces}, $\E$ cannot have Hilbert polynomial $2m + 1$, $2n + 1$, or $3m + 1$.
If $P_{\E} = n + 1$, then resolution (\ref{(r,1)}) with $r = 0$ fits into the commutative diagram
\[
\xymatrix
{
0 \ar[r] & \O(-1, 0) \ar[r]^-{\psi} \ar[d]^-{\beta} & \O \ar[r] \ar[d]^-{\alpha} & \E \ar[r] \ar[d] & 0 \\
0 \ar[r] & 2 \O(-1, -2) \ar[r]^-{\f} & \O(0, -1) \oplus \O \ar[r] & \G \ar[r] & 0
}
\]
with $\alpha \neq 0$. Since $\beta = 0$ we get $\alpha \psi = 0$, hence $\psi = 0$, which yields a contradiction.
We obtain a contradiction in the same manner if $P_{\E} = m+1$, $m + n + 1$, $m + 2n + 1$.
Assume that $P_{\E} = 2m + n + 1$. Then resolution (\ref{(r,1)}) with $r = 2$ is part of the commutative diagram
\[
\xymatrix
{
0 \ar[r] & \O(-1, -2) \ar[r]^-{\psi} \ar[d]^-{\beta} & \O \ar[r] \ar[d]^-{\alpha} & \E \ar[r] \ar[d] & 0 \\
0 \ar[r] & 2 \O(-1, -2) \ar[r]^-{\f} & \O(0, -1) \oplus \O \ar[r] & \G \ar[r] & 0
}
\]
Since $\alpha_{11} = 0$ we obtain $\f_{11} \beta_{11} + \f_{12} \beta_{21} = 0$.
This contradicts the fact that $\f_{11}$ and $\f_{12}$ are linearly independent.
Assume that $P_{\E} = 3m + n + 1$. Then resolution (\ref{(r,1)}) with $r = 3$ is the first line of the commutative diagram
\[
\xymatrix
{
0 \ar[r] & \O(-1, -3) \ar[r]^-{\psi} \ar[d]^-{\beta} & \O \ar[r] \ar[d]^-{\alpha} & \E \ar[r] \ar[d] & 0 \\
0 \ar[r] & 2 \O(-1, -2) \ar[r]^-{\f} & \O(0, -1) \oplus \O \ar[r] & \G \ar[r] & 0
}
\]
Write $\beta_{11} = 1 \tensor l_1$, $\beta_{21} = 1 \tensor l_2$,
$\f_{11} = x \tensor u_1 + y \tensor v_1$, $\f_{12} = x \tensor u_2 + y \tensor v_2$.
From $\alpha_{11} = 0$ we obtain
\begin{align*}
0 & = \f_{11} (1 \tensor l_1) + \f_{12} (1 \tensor l_2), \\
0 & = x \tensor (l_1 u_1 + l_2 u_2) + y \tensor (l_1 v_1 + l_2 v_2), \\
0 & = l_1 u_1 + l_2 u_2, \quad 0 = l_1 v_1 + l_2 v_2, \\
u_1 & = a l_2, \quad u_2 = - a l_1, \quad v_1 = b l_2, \quad v_2 = - b l_1, \\
\f_{11} & = (ax + by) \tensor l_2, \quad \f_{12} = - (ax + by) \tensor l_1
\end{align*}
for some $a, b \in \CC$.
This contradicts our hypothesis that $\f_{11}$ and $\f_{12}$ define a zero-dimensional subscheme of $\PP^1 \times \PP^1$.
Assume, finally, that $P_{\E} = 2m + 2n + 1$. Then resolution (\ref{(2,2)}) fits into the commutative diagram
\[
\xymatrix
{
0 \ar[r] & \O(-2, -1) \oplus \O(-1, -2) \ar[r]^-{\psi} \ar[d]^-{\beta} & \O(-1, -1) \oplus \O \ar[r] \ar[d]^-{\alpha} & \E \ar[r] \ar[d] & 0 \\
0 \ar[r] & 2 \O(-1, -2) \ar[r]^-{\f} & \O(0, -1) \oplus \O \ar[r] & \G \ar[r] & 0
}
\]
We have $\alpha_{22} \neq 0$ because the map $\E \to \G$ is injective on global sections.
It follows that $\alpha$ is injective, otherwise $\Ker(\alpha) \simeq \O(-1, -1)$, but this cannot be a subsheaf of $\O(-2, -1) \oplus \O(-1, -2)$.
It follows that $\beta$ is injective, which is absurd.
\end{proof}

\begin{corollary}
\label{rationality}
The variety $\MM$ is rational.
\end{corollary}

\begin{proof}
Consider the open subset $B \subset \MM_0$ given by the condition that $Z$ consist of two distinct points.
Notice that $B$ is a bundle with fiber $\PP^9$
and base an open subset of $((\PP^1 \times \PP^1)^2 \setminus \Delta)/\operatorname{S}_2$.
Here $\Delta$ is the diagonal of the product of two copies of $\PP^1 \times \PP^1$
and $\operatorname{S}_2$ is the group of permutations of two elements.
\end{proof}

\begin{proposition}
\label{sheaves_M_1}
Let $\F$ be an extension as in (\ref{generic_extension}), that has no zero-dimensional torsion,
for a curve $C$ of bidegree $(2, 3)$ and a subscheme $Z \subset C$
that is the intersection of two curves of bidegree $(0, 2)$, respectively, $(1, 0)$.
Then $\F$ gives a point in $\MM$.
Let $\MM_1 \subset \MM$ be the subset of such sheaves $\F$.
Then $\MM_1$ is irreducible of codimension $1$
and it can be described as the set of sheaves $\G$ having a resolution of the form
\begin{equation}
\label{M_1}
0 \lra \O(-2, -1) \oplus \O(-1, -3) \stackrel{\f}{\lra} \O(-1, -1) \oplus \O \lra \G \lra 0
\end{equation}
where $\f_{11} \neq 0$ and $\f_{12} \neq 0$.
\end{proposition}

\begin{proof}
We will show that any sheaf $\G$ having resolution (\ref{M_1}) has no destabilizing subsheaves.
Assume that $\G$ had a destabilizing subsheaf $\E$. Without loss of generality we may take $\E$ to be semi-stable.
Since $\dim \H^0(\G) = 1$, we have $\chi(\E) = 1$.
According to Corollary \ref{empty_spaces}, $\E$ cannot have Hilbert polynomial $2m + 1$, $2n + 1$, or $3m + 1$.
If $P_{\E} = n+1$, then resolution (\ref{(r,1)}) with $r = 0$ fits into the commutative diagram
\[
\xymatrix
{
0 \ar[r] & \O(-1, 0) \ar[r]^-{\psi} \ar[d]^-{\beta} & \O \ar[r] \ar[d]^-{\alpha} & \E \ar[r] \ar[d] & 0 \\
0 \ar[r] & \O(-2, -1) \oplus \O(-1, -3) \ar[r]^-{\f} & \O(-1, -1) \oplus \O \ar[r] & \G \ar[r] & 0
}
\]
with $\alpha \neq 0$. Since $\beta = 0$ we get $\alpha \psi = 0$, hence $\psi = 0$, which yields a contradiction.
We obtain a contradiction in the same manner if $P_{\E} = m+1$, $m + n + 1$, $2m + n + 1$.
Assume that $P_{\E} = m + 2n + 1$. Then resolution (\ref{(1,s)}) with $s = 2$ is part of the commutative diagram
\[
\xymatrix
{
0 \ar[r] & \O(-2, -1) \ar[r]^-{\psi} \ar[d]^-{\beta} & \O \ar[r] \ar[d]^-{\alpha} & \E \ar[r] \ar[d] & 0 \\
0 \ar[r] & \O(-2, -1) \oplus \O(-1, -3) \ar[r]^-{\f} & \O(-1, -1) \oplus \O \ar[r] & \G \ar[r] & 0
}
\]
Since $\alpha$ is injective on global sections, $\alpha$ is injective, hence $\beta$ is injective, too, and hence we may write
\[
\alpha = \left[
\begin{array}{c}
0 \\ 1
\end{array}
\right], \qquad \beta = \left[
\begin{array}{c}
1 \\ 0
\end{array}
\right]. \qquad \text{Then} \quad \f \beta = \left[
\begin{array}{l}
\f_{11} \\ \f_{21}
 \end{array}
 \right] = \alpha \psi = \left[
 \begin{array}{c}
 0 \\ \psi
 \end{array}
 \right],
 \]
 hence $\f_{11} = 0$, which contradicts our hypothesis. We obtain a contradiction in the same manner if $P_{\E} = 3m + n + 1$.
 Assume, finally, that $P_{\E} = 2m + 2n + 1$. Then resolution (\ref{(2,2)}) is the first line of the commutative diagram
\[
\xymatrix
{
0 \ar[r] & \O(-2, -1) \oplus \O(-1, -2) \ar[r]^-{\psi} \ar[d]^-{\beta} & \O(-1, -1) \oplus \O \ar[r] \ar[d]^-{\alpha} & \E \ar[r] \ar[d] & 0 \\
0 \ar[r] & \O(-2, -1) \oplus \O(-1, -3) \ar[r]^-{\f} & \O(-1, -1) \oplus \O \ar[r] & \G \ar[r] & 0
}
\]
Notice that $\alpha$ and $\alpha(1,1)$ are injective on global sections, hence $\alpha$ is injective, and hence $\beta$ is injective,
which is absurd.
\end{proof}

\noindent
Let $W_1$ be the set of morphisms $\f$ occurring in resolution (\ref{M_1}) and consider the algebraic group
\[
G_1 = \big( \Aut(\O(-2, -1) \oplus \O(-1, -3)) \times \Aut(\O(-1, -1) \oplus \O) \big)/\CC^*
\]
acting on $W_1$ by conjugation.

\begin{proposition}
\label{geometric_quotient}
The variety $\MM_1$ is isomorphic to the geometric quotient $W_1/G_1$.
Thus, $\MM_1$ is a $\PP^9$-bundle over $\PP^1 \times \PP^2$, so it is smooth and closed in $\MM$.
\end{proposition}

\begin{proof}
The canonical map $W_1 \to \MM_1$, $\f \mapsto [\Coker(\f)]$, has local sections, and its fibers are the $G_1$-orbits,
hence it is a geometric quotient map.
We construct the local sections as follows.
Given $[\F] \in \MM_1$, let $C$ be the schematic support of $\F$, and let $Z$ be the zero-dimensional scheme of length $2$
given by the exact sequence (\ref{generic_extension}).
Then $Z = L \cap (L_1 \cup L_2)$, where $L$ is a line of bidegree $(1, 0)$, and $L_1$, $L_2$ are lines, each of bidegree $(0, 1)$.
Choose equations $\f_{11} = 0$ of $L$, $\f_{12} = 0$ of $L_1 \cup L_2$, and $f = 0$ of $C$.
Then we can write $f = \f_{11} \f_{22} - \f_{12} \f_{21}$ for some $\f_{21} \in \Sym^2 V_1^* \tensor V_2^*$, $\f_{22} \in V_1^* \tensor \Sym^3 V_2^*$.
Map $[\F]$ to the morphism represented by the matrix $(\f_{ij})_{1 \le i, j \le 2}$.
This construction can be done for a local flat family in a neighborhood of $[\F]$ in $\MM_1$.

We now describe $W_1/G_1$. Let $U \subset V_1^* \oplus \Sym^2 V_2^*$ be the open subset
\[
\{ (\f_{11}, \f_{12}),\, \f_{11} \neq 0, \f_{12} \neq 0 \}.
\]
Let $F$ be the trivial vector bundle on $U$ with fiber $(\Sym^2 V_1^* \tensor V_2^*) \oplus (V_1^* \tensor \Sym^3 V_2^*)$.
Consider the subbundle $E \subset F$ which over the point $(\f_{11}, \f_{12})$ has fiber
$(\f_{11} V_1^* \tensor V_2^*) \oplus (\f_{12} V_1^* \tensor V_2^*)$.
The quotient bundle $G = F/E$ has rank $10$ and is linearized for the canonical action of
$\CC^* \times \CC^* = \Aut(\O(-2, -1) \oplus \O(-1, -3))$ on $U$. Thus, $G$ descends to a vector bundle $H$ over
$U/\CC^* \times \CC^* = \PP(V_1^*) \times \PP(\Sym^2 V_2^*) \simeq \PP^1 \times \PP^2$.
Clearly, $\PP(H) \simeq W_1/G_1$.
\end{proof}

\begin{proposition}
\label{sheaves_M_3}
Assume that $\F$ gives a point in $\MM$ and that $\H^0(\F(0, -1)) = 0$.
Assume that the maximal minors of $\f_1$ have a common factor.
Then $\Ker(\f_1) \simeq \O(-2, -2)$ and $\Coker(\f_1) \simeq \O_L$ for a line $L \subset \PP^1 \times \PP^1$ of bidegree $(0, 1)$.
Thus, we have an extension
\begin{equation}
\label{M_3}
0 \lra \O_Q \lra \F \lra \O_L \lra 0
\end{equation}
for a quartic curve $Q \subset \PP^1 \times \PP^1$ of bidegree $(2, 2)$.
Conversely, any non-split extension of this form is semi-stable.
We have $\Ext^1_{\O_{\PP^1 \times \PP^1}}(\O_L, \O_Q) \simeq \CC^2$.
\end{proposition}

\begin{proof}
Let $g = \gcd (\zeta_1, \zeta_2, \zeta_3, \zeta_4)$,
where $\zeta_1$, $\zeta_2$, $\zeta_3$, $\zeta_4$ are defined as in the proof of Lemma \ref{generic_sheaves}.
We have the exact sequence
\[
0 \lra \O(i, j) \stackrel{\eta}{\lra} 4\O(-1, -1) \stackrel{\f_1}{\lra} \O(0, -1) \oplus 2\O(-1, 0),
\]
\[
\eta = \left[
\begin{array}{cccc}
\frac{\zeta_1}{g} & -\frac{\zeta_2}{g} & \phantom{-} \frac{\zeta_3}{g} & - \frac{\zeta_4}{g}
\end{array}
\right]^\trans.
\]
The possibilities for the kernel of $\f_1$ are given in Table 3 below.

\begin{table}[ht]{Table 3. Kernel of $\f_1$.}
\begin{center}
\begin{tabular}{| c | c | c |}
\hline
$\deg (g)$ & $(i, j)$ & $P_{\Coker(\f_5)}$
\\
\hline
$(1, 0)$ & $(-1, -3)$ & $3m + n  + 1$
\\
$(0, 1)$ & $(-2, -2)$ & $2m + 2n$
\\
$(0, 2)$ & $(-2, -1)$ & $m + 2n + 1$
\\
$(1, 1)$ & $(-1, -2)$ & $2m + n + 1$
\\
\hline
\end{tabular}
\end{center}
\end{table}

\noindent
We see that the only case in which $\Coker(\f_5)$ does not destabilize $\F$ is the case $(i, j) = (-2, -2)$.
Thus, $\Ker(\f_1) \simeq \O(-2, -2)$. The cokernel of $\f_1$ has no zero-dimensional torsion and has Hilbert polynomial $m + 1$,
hence it is of the form $\O_L$ for a line $L$ of bidegree $(0, 1)$.
From sequence (\ref{generic_convergence}) we see that $\F$ is an extension of $\O_L$ by $\O_Q$.

Conversely, assume that $\F$ is such an extension. By Proposition (\ref{O_C_stable}) $\O_Q$ is stable.
Thus, for any proper subsheaf $\E \subset \F$ we have $\p(\E \cap \O_Q) < 0$ unless $\O_Q \subset \E$.
Since, obviously, $\O_L$ is stable, the image of $\E$ in $\O_L$ has slope at most $1$.
It follows that $\p(\E) < \p(\F)$, hence $\F$ is stable.
From the short exact sequence
\[
0 \lra \O(0, -1) \lra \O \lra \O_L \lra 0
\]
we get the long exact sequence
\begin{align*}
0 = \operatorname{Hom}(\O_L, \O_Q) \lra & \H^0(\O_Q) \simeq \CC \lra \H^0(\O_Q(0, 1)) \simeq \CC^2 \\
\lra \Ext^1(\O_L, \O_Q) \lra & \H^1(\O_Q) \simeq \CC \lra \H^1(\O_Q(0, 1)) = 0.
\end{align*}
This proves that $\Ext^1(\O_L, \O_Q) \simeq \CC^2$.
\end{proof}

\noindent
Let $\MM_2 \subset \MM$ be the subset of sheaves having resolution (\ref{M_2}).
Clearly, $\MM_2 \simeq \PP^{11}$.
Let $\MM_3 \subset \MM$ be the subset of extension sheaves $\F$ as in (\ref{M_3}).
Clearly, $\MM_3$ is a bundle with base $\PP^8 \times \PP^1$ and fiber $\PP^1$.
Thus, $\MM_3$ is closed of codimension $3$.
It intersects $\MM_2$ along a subvariety isomorphic to $\PP^8 \times \PP^1$ consisting of twisted structure sheaves
$\O_C(0, 1)$, where $C = Q \cup L$.
The subvarieties $\MM_0$, $\MM_1$, $\MM_2 \cup \MM_3$ form a decomposition of $\MM$
and satisfy the properties from Theorem \ref{main_theorem}.

%%%%%%%%%%%%%%%%%%%%%%%%%%%%%%%%%%%%%%%%%%%%%%%%%%%%%%%%

\section{Variation of moduli of $\alpha$-semi-stable pairs}
\label{variation}

Let $X$ be a separated scheme of finite type over $\CC$.
An \emph{algebraic system} on $X$ is a triple $\Lambda = (\Gamma, \sigma, \F)$ consisting of an $\O_X$-module $\F$,
a vector space $\Gamma$ over $\CC$, and a $\CC$-linear map $\sigma \colon \Gamma \to \H^0(\F)$.
If $\F$ is a coherent $\O_X$-module and $\Gamma$ is finite dimensional, we say that $\Lambda$ is a \emph{coherent system}.
A \emph{pair} will be a coherent system in which $\sigma$ is injective and $\dim \Gamma = 1$.
A morphism of algebraic systems $(\gamma, \f) \colon (\Gamma, \sigma, \F) \to (\Gamma', \sigma', \F')$
consists of a $\CC$-linear map $\gamma \colon \Gamma \to \Gamma'$ together with a morphism of $\O_X$-modules
$\f \colon \F \to \F'$, which are compatible, in the sense that $\H^0(\f) \sigma = \sigma' \gamma$.
These notions were introduced in \cite{lepotier_asterisque} and \cite{he} where appropriate semi-stability conditions
of coherent systems were defined, which led in a natural manner to the construction of moduli spaces.
The category of algebraic systems on $X$ is abelian and, according to \cite[Th\'eor\`eme 1.3]{he}, it has enough injectives.
Thus, we can define the left derived functors of $\Hom(\Lambda, -)$, denoted $\Ext^i(\Lambda, -)$.
Our basic tool for computing these extension spaces is \cite[Corollaire 1.6]{he}, which we quote below.

\begin{proposition}
\label{ext_sequence}
Let $\Lambda = (\Gamma, \sigma, \F)$ and $\Lambda' = (\Gamma', \sigma', \F')$ be two algebraic systems on $X$ with
$\sigma'$ injective. Then there is a long exact sequence
\begin{align*}
0 & \lra \Hom(\Lambda, \Lambda') \lra \Hom(\F, \F') \lra \Hom(\Gamma, \H^0(\F')/\Gamma') \\
& \lra \Ext^1(\Lambda, \Lambda') \lra \Ext^1(\F, \F') \lra \Hom(\Gamma, \H^1(\F')) \\
& \lra \Ext^2(\Lambda, \Lambda') \lra \Ext^2(\F, \F') \lra \Hom(\Gamma, \H^2(\F')).
\end{align*}
\end{proposition}

From now on we specialize to the case when $X = \PP^1 \times \PP^1$ with fixed polarization $\O(1, 1)$,
and $\F$ has dimension $1$ with Hilbert polynomial $P_{\F}(m, n) = rm + sn + t$.
Let $\alpha$ be a positive rational number.
We define the \emph{slope} of a coherent system $\Lambda = (\Gamma, \sigma, \F)$ relative to $\alpha$ and to the fixed
polarization
\[
\p_{\alpha}(\Lambda) = \frac{\dim \Gamma}{r + s} \alpha + \frac{t}{r + s}.
\]
We say that $\Lambda$ is \emph{$\alpha$-semi-stable} (respectively \emph{$\alpha$-stable})
if $\F$ has no zero-dimensional torsion, $\sigma$ is injective, and for any proper coherent subsystem $\Lambda' \subset \Lambda$
we have $\p_{\alpha}(\Lambda') \le \p_{\alpha}(\Lambda)$ (respectively $\p_{\alpha}(\Lambda') < \p_{\alpha}(\Lambda)$).
According to \cite{he}, for fixed polynomial $P$ and $\alpha \in {\mathbb Q}_{> 0}$ there is a coarse moduli space
$\operatorname{Syst}_{X, \alpha}(P)$ parametrizing S-equivalence classes of $\alpha$-semi-stable coherent systems
$(\Gamma, \F)$ such that $P_{\F} = P$. We have a decomposition of $\operatorname{Syst}_{X, \alpha}(P)$
into disjoint components according to $\dim \Gamma$.
The component corresponding to the case $\dim \Gamma = 1$, i.e. parametrizing $\alpha$-semi-stable pairs
with fixed Hilbert polynomial $P$, will be denoted $\M^{\alpha}(P)$.

A value $\alpha_0$ is said to be \emph{regular} relative to $P$ if it is contained in an interval $(\alpha_1, \alpha_2)$
such that the set of $\alpha$-semi-stable pairs with Hilbert polynomial $P$ remains unchanged as $\alpha$ varies in
$(\alpha_1, \alpha_2)$.
If there is no such interval we say that $\alpha_0$ is a \emph{wall} relative to $P$.
The following proposition is analogous to \cite[Lemma 3.1]{choi_chung}.

\begin{proposition}
\label{walls}
Relative to $P(m, n) = 3m + 2n + 1$ we have only one wall at $\alpha = 4$.
\end{proposition}

\begin{proof}
According to the proof of \cite[Th\'eor\`eme 4.2]{he}, $\alpha$ is a wall if and only if there is a strictly $\alpha$-semi-stable
pair $\Lambda = (\Gamma, \F)$. There is a pair $\Lambda' = (\Gamma', \F') \neq \Lambda$ which is a subpair of $\Lambda$
or a quotient pair such that $\p_{\alpha}(\Lambda') = \p_{\alpha}(\Lambda)$.
Write $P_{\F'}(m, n) = rm + sn + t$ with $r \le 3$, $s \le 2$.
We have the equation
\begin{equation}
\label{alpha}
\frac{\alpha + t}{r + s} = \frac{\alpha + 1}{5}.
\end{equation}
Without loss of generality we may assume that $\Gamma'$ generates $\F'$ away, possibly, from finitely many points.
Thus $t \ge r + s - rs$.
The case when $r = 3$, $s = 2$ is unfeasible. Assume that $r = 2$, $s = 2$, $t \ge 0$.
Equation (\ref{alpha}) becomes $\alpha = 4 - 5 t$, which has solution $\alpha = 4$ when $t = 0$.
For all other choices of $r$ and $s$ we have $t \ge 1$, hence equation (\ref{alpha}) has no positive solution.
\end{proof}

\noindent
We write $\MM^{\alpha} = \M^{\alpha}(3m + 2n + 1)$. The moduli spaces $\MM^{\alpha}$ remain unchanged
as $\alpha$ varies in the interval $(0, 4)$ and will be denoted $\MM^{0+}$.
Likewise, for $\alpha \in (4, \infty)$, $\MM^{\alpha}$ are all equal to a moduli space denoted $\MM^{\infty}$.
These moduli spaces are related by the flipping diagram
\[
\xymatrix
{
\MM^{\infty} \ar[dr]_-{\rho_{\infty}} & & \MM^{0+} \ar[ld]^-{\rho_0} \\
& \MM^4
}
\]
in which the maps $\rho_{\infty}$ and $\rho_0$ are induced by the inclusion of sets of $\alpha$-semi-stable pairs.
In particular, $\rho_{\infty}$ and $\rho_0$ are birational.

The following proposition is a particular case of \cite[Proposition B.8]{pandharipande_thomas}.

\begin{proposition}
\label{M^infinity}
The variety $\MM^{\infty}$ is isomorphic to the flag Hilbert scheme of
zero-dimensional subschemes of length $2$ contained in curves of bidegree $(2, 3)$ in $\PP^1 \times \PP^1$.
\end{proposition}

\noindent
In particular, $\MM^{\infty}$ is a bundle with base $\Hilb_{\PP^1 \times \PP^1}(2)$ and fiber $\PP^9$, so it is smooth.
This proposition gives another proof for the fact that $\MM$ is rational (Corollary \ref{rationality}).

\begin{remark}
\label{flipping_base}
From the proof of Proposition \ref{walls}, we see that the S-equivalence type of a strictly $\alpha$-semi-stable pair in
$\MM^4$ is of the form $(\Gamma, \E) \oplus (0, \O_L)$, where $(\Gamma, \E) \in \M^{0+}(2m + 2n)$
and $L \subset \PP^1 \times \PP^1$ is a line of bidegree $(0, 1)$.
As in the proof of Proposition \ref{vanishing_M}, $\E$ has a subsheaf isomorphic to the structure sheaf of a curve.
By semi-stability, the curve must have bidegree $(2, 2)$.
We see that $\E \simeq \O_Q$ for a quartic curve $Q \subset \PP^1 \times \PP^1$ of bidegree $(2, 2)$.
Thus, $\M^{0+}(2m + 2n) \simeq \PP^8$.
\end{remark}

Let $F^{\infty} \subset \MM^{\infty}$ and $F^0 \subset \MM^{0+}$ be the flipping loci, that is, the inverse images under
$\rho_{\infty}$, respectively, under $\rho_0$ of $\M^{0+}(2m + 2n) \times \M(m + 1)$.
The fiber of $F^{\infty}$ over $(\Lambda_1, \Lambda_2)$ is $\PP(\Ext^1(\Lambda_1, \Lambda_2))$.
The fiber of $F^0$ over $(\Lambda_1, \Lambda_2)$ is $\PP(\Ext^1(\Lambda_2, \Lambda_1))$.

\begin{remark}
\label{flipping_loci}
The flipping locus $F^{\infty}$ is a projective bundle with fiber $\PP^2$ and base $\M^{0+}(2m + 2n) \times \M(m + 1)$.
The flipping locus $F^0$ is a $\PP^1$-bundle with the same base.
Indeed, take $\Lambda_1 = (\Gamma, \O_Q) \in \M^{0+}(2m + 2n)$ and $\Lambda_2 = (0, \O_L) \in \M(m+1)$.
Proposition \ref{ext_sequence} yields the exact sequence
\begin{align*}
0 & \lra \Hom(\Lambda_1, \Lambda_2) \lra \Hom(\O_Q, \O_L) \lra \Hom(\Gamma, \H^0(\O_L)) \simeq \CC \\
& \lra \Ext^1(\Lambda_1, \Lambda_2) \lra \Ext^1(\O_Q, \O_L) \lra \Hom(\Gamma, \H^1(\O_L)) = 0.
\end{align*}
Any morphism $\Lambda_1 \to \Lambda_2$ is zero because $\Gamma = \H^0(\O_Q)$ generates $\O_Q$.
If $L \nsubseteq Q$, then $\Hom(\O_Q, \O_L) = 0$; if $L \subset Q$, then $\Hom(\O_Q, \O_L) \simeq \CC$.
From the short exact sequence
\[
0 \lra \O(-2, -2) \lra \O \lra \O_Q \lra 0
\]
we get the long exact sequence
\begin{align*}
0 & \lra \Hom(\O_Q, \O_L) \lra \H^0(\O_L) \simeq \CC \lra \H^0(\O_L(2, 2)) \simeq \CC^3 \\
& \lra \Ext^1(\O_Q, \O_L) \lra \H^1(\O_L) = 0.
\end{align*}
Thus, if $L \nsubseteq Q$, then $\Ext^1(\O_Q, \O_L) \simeq \CC^2$; if $L \subset Q$, then $\Ext^1(\O_Q, \O_L) \simeq \CC^3$.
In either case we get $\Ext^1(\Lambda_1, \Lambda_2) \simeq \CC^3$.

We will now verify the isomorphism $\Ext^1(\Lambda_2, \Lambda_1) \simeq \CC^2$.
From Proposition \ref{ext_sequence} we have the exact sequence
\[
0 = \Hom(0, \H^0(\O_Q)/\Gamma) \to \Ext^1(\Lambda_2, \Lambda_1) \to \Ext^1(\O_L, \O_Q) \to \Hom(0, \H^1(\O_Q)) = 0
\]
Thus, the middle arrow is an isomorphism.
From Proposition \ref{sheaves_M_3} we know that $\Ext^1(\O_L, \O_Q) \simeq \CC^2$.
\end{remark}

\begin{lemma}
\label{ext^2}
For $\Lambda \in F^0$ we have $\Ext^2(\Lambda, \Lambda) = 0$.
\end{lemma}

\begin{proof}
We have a non-split exact sequence
\[
0 \lra \Lambda_1 \lra \Lambda \lra \Lambda_2 \lra 0
\]
for some $\Lambda_1 = (\Gamma, \O_Q) \in \M^{0+}(2m + 2n)$ and $\Lambda_2 = (0, \O_L) \in \M(m + 1)$.
It is enough to show that $\Ext^2(\Lambda_i, \Lambda_j) = 0$ for $i, j = 1, 2$.
From Proposition \ref{ext_sequence} we have the exact sequence
\[
0 = \Hom(\Gamma, \H^1(\O_L)) \lra \Ext^2(\Lambda_1, \Lambda_2) \lra
\Ext^2(\O_Q, \O_L) \simeq \Hom(\O_L, \O_Q \tensor \omega)^*.
\]
The group on the right vanishes because $\O_L$ is stable, by Proposition \ref{O_C_stable} $\O_Q \tensor \omega$ is stable
and $\p(\O_L) > \p(\O_Q \tensor \omega)$. Thus, $\Ext^2(\Lambda_1, \Lambda_2) = 0$.
The exact sequence
\[
0 = \Hom(0, \H^1(\O_Q)) \to \Ext^2(\Lambda_2, \Lambda_1)
\lra \Ext^2(\O_L, \O_Q) \simeq \Hom(\O_Q, \O_L \tensor \omega)^* = 0
\]
shows that $\Ext^2(\Lambda_2, \Lambda_1) = 0$.
We have the exact sequence
\begin{align*}
0 = \Hom(\Gamma, & \H^0(\O_Q)/\Gamma) \\
\lra & \Ext^1(\Lambda_1, \Lambda_1) \lra \Ext^1(\O_Q, \O_Q) \lra
\Hom(\Gamma, \H^1(\O_Q)) \simeq \CC \\
\lra & \Ext^2(\Lambda_1, \Lambda_1) \lra \Ext^2(\O_Q, \O_Q) \simeq \Hom(\O_Q, \O_Q \tensor \omega)^* = 0.
\end{align*}
The space $\Ext^1(\Lambda_1, \Lambda_1)$ is isomorphic to the tangent space of $\M^{0+}(2m+2n) \simeq \PP^8$
at $\Lambda_1$, so it is isomorphic to $\CC^8$.
From the short exact sequence
\[
0 \lra \O(-2, -2) \lra \O \lra \O_Q \lra 0
\]
we get the long exact sequence
\begin{align*}
0 \lra \Hom(\O_Q, \O_Q) \simeq \CC & \lra \H^0(\O_Q) \simeq \CC \lra \H^0(\O_Q(2,2)) \simeq \CC^8 \\
\lra \Ext^1(\O_Q, \O_Q) & \lra \H^1(\O_Q) \simeq \CC \lra \H^1(\O_Q(2, 2)) = 0.
\end{align*}
Thus $\Ext^1(\O_Q, \O_Q) \simeq \CC^9$. We get the vanishing of $\Ext^2(\Lambda_1, \Lambda_1)$.
Finally, from the exact sequence
\[
0 = \Hom(0, \H^1(\O_L)) \lra \Ext^2(\Lambda_2, \Lambda_2) \lra \Ext^2(\O_L, \O_L) \simeq \Hom(\O_L, \O_L \tensor \omega)^* = 0
\]
we get the vanishing of $\Ext^2(\Lambda_2, \Lambda_2)$.
\end{proof}

\noindent
The following theorem is analogous to \cite[Theorem 3.3]{choi_chung}.

\begin{theorem}
\label{wall_crossing}
Let $\MM^{\alpha}$ be the moduli space of $\alpha$-semi-stable pairs on $\PP^1 \times \PP^1$ with Hilbert polynomial
$P(m, n) = 3m + 2n + 1$.
We have the following commutative diagram expressing the variation of $\MM^{\alpha}$ as $\alpha$ crosses the wall:
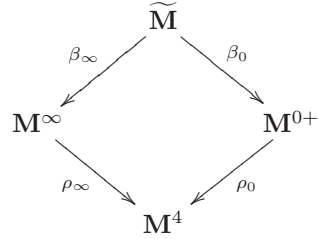
\[
\xymatrix
{
& \widetilde{\MM} \ar[dl]_-{\beta_{\infty}} \ar[dr]^-{\beta_0} \\
\MM^{\infty} \ar[dr]_-{\rho_{\infty}} & & \MM^{0+} \ar[ld]^-{\rho_0} \\
& \MM^4
}
\]
Here $\beta_{\infty}$ is the blow-up with center $F^{\infty}$ and $\beta_0$ is the blow-down contracting the exceptional divisor
$\widetilde{F}$ in the direction of $\PP^2$, where we regard $\widetilde{F}$ as a $\PP^2 \times \PP^1$-bundle over
$\M^{0+}(2m + 2n) \times \M(m + 1)$.
\end{theorem}

\begin{proof}
At \cite[Theorem 3.3]{choi_chung} a birational map $\beta_0$ is constructed
from the blow-up $\widetilde{\MM}$ of $\MM^{\infty}$ along $F^{\infty}$ to $\MM^{0+}$,
which contracts $\widetilde{F}$ in the $\PP^2$-directions.
Note that $\beta_0$ gives an isomorphism on the complement of $F^0$
and the preimages of points in $F^0$ are isomorphic to $\PP^2$.
By Remark \ref{flipping_loci}, $F^0$ is smooth.
We claim that $\MM^{0+}$ is also smooth.
This can be verified using the smoothness criterion for moduli spaces of $\alpha$-semi-stable pairs:
if $\Lambda$ gives a stable point of $\MM^{0+}$ and $\Ext^2(\Lambda, \Lambda) = 0$, then $\Lambda$ gives a smooth point.
It is enough to take $\Lambda \in F^0$ and then we can apply Lemma \ref{ext^2}.
We can now apply the Universal Property of the blow-up \cite[p. 604]{griffiths_harris}, to conclude that $\beta_0$ is a blow-up
with center $F^0$ and exceptional divisor $\widetilde{F}$.
\end{proof}

\noindent
The following proposition is analogous to \cite[Proposition 4.4]{choi_chung}. We define the \emph{forgetful morphism}
$\phi \colon \MM^{0+} \to \MM$ by $\phi (\Gamma, \F) = [\F]$.

\begin{proposition}
\label{contraction}
The forgetful morphism $\phi \colon \MM^{0+} \to \MM$ is a blow-up of $\MM$ along $\MM_2$.
\end{proposition}

\begin{proof}
We will give a simpler argument then the one found at \cite[Proposition 4.4]{choi_chung}.
As seen in the proof of Theorem \ref{wall_crossing}, $\MM^{0+}$ is smooth.
The varieties $\MM$ and $\MM_2$ are also smooth.
Away from $\MM_2$, $\phi$ is an isomorphism because, by Theorem \ref{main_theorem}, for $\F \in \MM \setminus \MM_2$
we have $\H^0(\F) \simeq \CC$, hence we may identify $\F$ with the $\alpha$-stable pair $(\H^0(\F), \F)$ for sufficiently small $\alpha$.
For $\F \in \MM_2$, $\phi^{-1}([\F]) = \PP(\H^0(\F)) \simeq \PP^1$.
By the Universal Property of the blow-up \cite[p. 604]{griffiths_harris}, $\phi$ is a blow-up with center $\MM_2$.
\end{proof}

\noindent
{\bf Proof of Theorem \ref{poincare_polynomial}.}
The integral homology groups of $\MM$ have no torsion because $\MM^{\infty}$ enjoys this property and $\MM$ is obtained
from $\MM^\infty$ by a sequence of blow-ups and blow-downs.
By Theorem \ref{wall_crossing},
\[
\Poly(\MM^{0+}) = \Poly(\MM^{\infty}) + (\Poly(\PP^1) - \Poly(\PP^2)) \Poly (\M^{0+}(2m + 2n) \times \M(m + 1)).
\]
By Proposition \ref{M^infinity} and Remark \ref{flipping_base},
\[
\Poly(\MM^{0+}) = \Poly(\PP^9) \Poly(\Hilb_{\PP^1 \times \PP^1}(2)) + (\Poly(\PP^1) - \Poly(\PP^2)) \Poly(\PP^8) \Poly(\PP^1).
\]
According to \cite[Theorem 0.1]{goettsche},
\[
\Poly(\Hilb_{\PP^1 \times \PP^1}(2))(\xi) = \xi^4 + 3 \xi^3 + 6 \xi^2 + 3 \xi + 1.
\]
In view of Proposition \ref{contraction},
\[
\Poly(\MM) = \Poly(\MM^{0+}) - \xi \Poly(\MM_2) = \Poly(\MM^{0+}) - \xi \Poly(\PP^{11}).
\]
In conclusion,
\[
\Poly(\MM) = \frac{\xi^{10} - 1}{\xi - 1} (\xi^4 + 3 \xi^3 + 6 \xi^2 + 3 \xi + 1) - \xi^2 \frac{\xi^9 - 1}{\xi - 1} (\xi + 1) - \xi \frac{\xi^{12} - 1}{\xi - 1}.
\]

\noindent \\
{\bf Acknowledgements.}
The author would like to thank Jean-Marc Dr\'ezet for several helpful suggestions, especially concerning Lemma \ref{unique_extension}.

\end{document}